\documentclass[12pt]{amsart}
\usepackage{amsfonts}
\usepackage{amsfonts,latexsym,rawfonts,amsmath,amssymb,amsthm}
\usepackage[plainpages=false]{hyperref}
\usepackage{graphicx}

\numberwithin{equation}{section}

\RequirePackage{color}
\theoremstyle{plain}

\newtheorem{theorem}{Theorem}[section]

\newtheorem{corollary}{Corollary}[section]
\newtheorem{definition}{Definition}[section]

\newtheorem{lemma}[theorem]{Lemma}

\newtheorem{remark}[theorem]{Remark}
\newtheorem{problem}[theorem]{Problem}

\newcommand{\beq}{\begin{equation}}
\newcommand{\eeq}{\end{equation}}
\newcommand{\beqs}{\begin{eqnarray*}}
\newcommand{\eeqs}{\end{eqnarray*}}
\newcommand{\beqn}{\begin{eqnarray}}
\newcommand{\eeqn}{\end{eqnarray}}
\newcommand{\beqa}{\begin{array}}
\newcommand{\eeqa}{\end{array}}

\def\phi{\varphi}

\begin{document}
\title[Minkowski problem]{Non-normalized solutions to the horospherical Minkowski problem}

\author{Li Chen}
\address{Faculty of Mathematics and Statistics, Hubei Key Laboratory of Applied Mathematics, Hubei University,  Wuhan 430062, P.R. China}
\email{chenli@hubu.edu.cn}




\begin{abstract}
Recently, the horospherical $p$-Minkowski problem in hyperbolic space was proposed as a counterpart of
$L_p$ Minkowski problem in Euclidean space. Through designing a new volume preserving curvature flow,
the existence of normalized even solution to the horospherical $p$-Minkowski problem was solved for all $p \in \mathbb{R}$.

However, due to the lack of homogeneity of the horospherical $p$-surface area measure, it is
 difficult to remove the normalizing factor. In this paper, we overcome this difficulty
for $-n\leq p<n$ by the degree theory. In particular, our result gives the existence for solutions to
the prescribed surface area measure problem for
$h$-convex domains ($p=0$) and the prescribed
shifted Gauss curvature problem (or the $n$-th symmetric
Christoffel problem in $\mathbb{H}^{n+1}$) ($p=-n$).
\end{abstract}

\maketitle

\baselineskip18pt

\parskip3pt

\section{Introduction}

The classical Minkowski problem and its extension $L_p$ Minkowski problem \cite{Lu93, Chou-Wang06} play very important roles
in the study of convex bodies in Euclidean space (see also \cite{Sch13}). It is a dream to develop similar problems
in hyperbolic space. Recently, Andrews-Chen-Wei declared
interesting formal similarities between the geometry of
$h$-convex domains in hyperbolic space and that of
convex Euclidean bodies (see section 5 in \cite{And20}). Along the lines of Andrews-Chen-Wei, Li-Xu \cite{Li-Xu}
introduced the horospherical $p$-surface area measure by use of the hyperbolic $p$-sum which they defined
and proposed the associated horospherical $p$-Minkowski problem. We will briefly describe them below
followed by Section 5 in \cite{And20} and Section 2 in \cite{Li-Xu}.

We shall work
in the hyperboloid model of $\mathbb{H}^{n+1}$. For that, consider the
Minkowski space $\mathbb{R}^{n+1,1}$ with canonical coordinates
$X=(X^1, . . . , X^{n+1}, X^0)$ and the Lorentzian metric
\begin{eqnarray*}
\langle X, X\rangle=\sum_{i=1}^{n+1}(X^i)^2-(X^0)^2.
\end{eqnarray*}
$\mathbb{H}^{n+1}$ is the future time-like hyperboloid in Minkowski
space $\mathbb{R}^{n+1,1}$, i.e.
\begin{eqnarray*}
\mathbb{H}^{n+1}=\Big\{X=(X^1, \cdot\cdot\cdot, X^{n+1}, X^0) \in
\mathbb{R}^{n+1,1}: \langle X, X\rangle=-1, X^0>0\Big\}.
\end{eqnarray*}
The horospheres are hypersurfaces in $\mathbb{H}^{n+1}$ whose
 principal curvatures equal to $1$ everywhere. In the hyperboloid model of $\mathbb{H}^{n+1}$,
they can be parameterized by $\mathbb{S}^n \times \mathbb{R}$
\begin{eqnarray*}
H_{x}(r)=\{X \in \mathbb{H}^{n+1}: \langle X, (x, 1)\rangle=-e^r\},
\end{eqnarray*}
where $x \in \mathbb{S}^n$ and $r \in \mathbb{R}$ represents the signed geodesic distance from the ``north pole" $N =(0, 1)\in
\mathbb{H}^{n+1}$.
The interior of the horosphere is called the horo-ball and we denote by
\begin{eqnarray*}
B_{x}(r)=\{X \in \mathbb{H}^{n+1}: 0>\langle X, (x,
1)\rangle>-e^r\}.
\end{eqnarray*}
If we use the Poincar\'e ball model $\mathbb{B}^{n+1}$ of $\mathbb{H}^{n+1}$, then $B_x(r)$ corresponds to
an $(n+1)$-dimensional ball which tangents to $\partial \mathbb{B}^{n+1}$ at $x$. Furthermore, $B_x(r)$ contains the origin for $r>0$.

\begin{definition}
A compact domain $\Omega\subset \mathbb{H}^{n+1}$ (or its boundary $\partial \Omega$) is horospherically convex
(or $h$-convex for short) if every boundary point $p$ of
$\partial \Omega$ has a supporting horo-ball, i.e. a horo-ball $B$
such that $\Omega\subset \overline{B}$ and $p \in \partial B$.
When $\Omega$ is smooth, the $h$-convexity of $\Omega$ implies that the principal curvatures of $\partial \Omega$ are greater
than or equal to $1$.

For a smooth compact domain $\Omega$, we say
$\Omega$ (or $\partial \Omega$) is uniformly $h$-convex if the principal curvatures of $\partial \Omega$ are greater than $1$.
\end{definition}

\begin{definition}
Let $\Omega\subset \mathbb{H}^{n+1}$ be a compact and $h$-convex domain. For each $X \in \partial \Omega$,
$\partial \Omega$ has a supporting horo-ball $B_{x}(r)$
for some $r \in \mathbb{R}$ and $x \in \mathbb{S}^n$. Then its horospherical
Gauss map
$G: \partial \Omega\rightarrow \mathbb{S}^n$
is defined by
$
G(X)=x.
$
\end{definition}

Let $\Omega$ be a compact and $h$-convex domain in $\mathbb{H}^{n+1}$. Then
for each $x\in \mathbb{S}^n$ we define the horospherical support function
of $\Omega$ (or $\partial \Omega$) in direction $x$ by
\begin{eqnarray*}
u_{\Omega}(x):=\inf\{r \in \mathbb{R}: \Omega\subset \overline{B}_{x}(r)\}.
\end{eqnarray*}
We also have the alternative characterisation
\begin{eqnarray}\label{SD}
u_{\Omega}(x)=\sup\{\log(-\langle X, (x, 1)\rangle): X \in \Omega\}.
\end{eqnarray}
We will write $u(x)$ for $u_{\Omega}(x)$ if there is no confusion about the choice of the domain.

The support function completely determines a compact
$h$-convex domain $\Omega$, as an intersection of horo-balls:
\begin{eqnarray}\label{1-1}
\Omega=\bigcap_{x \in \mathbb{S}^n}\overline{B}_{x}(u_{\Omega}(x)).
\end{eqnarray}

If the compact domain $\Omega$ is uniformly
$h$-convex, then $G$ is a diffeomorphism from $\partial \Omega$ to $\mathbb{S}^n$. Thus, $\partial \Omega$ can be viewed
as a smooth embedding $\overline{X}=G^{-1}: \mathbb{S}^n\rightarrow \mathbb{R}^{n+1, 1}$ and
$\overline{X}$ can be written in terms of the support function $u$, as follows:
\begin{eqnarray}\label{X}
\overline{X}(x)=\frac{1}{2}\varphi(-x, 1)+\frac{1}{2}
\Big(\frac{|D \varphi|^2}{\varphi}+\frac{1}{\varphi}\Big)(x, 1)-(D\varphi, 0),
\end{eqnarray}
where $\varphi=e^u$ and
$D$ is the Levi-Civita connection of the standard metric $\sigma$ of $\mathbb{S}^{n}$.
Then, after choosing normal coordinates around $x$ on $\mathbb{S}^{n+1}$,
the area element $d\mu$ of $\Omega$ at $\overline{X}(x)=G^{-1}(x)$ can be given by
\begin{eqnarray}\label{Area}
d\mu=\sqrt{\det \langle \partial_i X, \partial_j X\rangle}d\sigma=\det A[\varphi]d\sigma,
\end{eqnarray}
where
\begin{eqnarray*}
A[\varphi]=D^2\varphi-\frac{1}{2}\frac{|D\varphi|^2}{\varphi}I+\frac{1}{2}\Big(\varphi-\frac{1}{\varphi}\Big)I
\end{eqnarray*}
and $I$ is the identity matrix.
Moreover, the compact domain $\Omega$ is uniformly
$h$-convex if and only if the matrix $A[\varphi]$ is positive definite.

Until now, we have seen many interesting similarities between the geometry of
$h$-convex domains in hyperbolic space and that of
convex Euclidean bodies.
Recently, Li-Xu \cite{Li-Xu} developed deeply this similarities.  In particular, they introduced a sum of
two sets in hyperbolic space which they called the hyperbolic $p$-sum (see Definition 1.1 \cite{Li-Xu}).

\begin{definition}
Let $\frac{1}{2}\leq p\leq 2$, $a\geq 0$, $b\geq 0$ and $a+b\geq 1$, and let $K$ and $L$ be two smooth uniformly
$h$-convex and compact domains with the horospherical support functions $u_K(x)$ and $u_L(x)$ respectively. The hyperbolic $p$-sum
$\Omega:=a\cdot K+_{p} b \cdot L$ of $K$ and $L$ is defined by the $h$-convex compact domain with the horospherical support function
\begin{eqnarray*}
u_{\Omega}(x):=\frac{1}{p}\log\Big(ae^{pu_K(x)}+be^{pu_L(x)}\Big).
\end{eqnarray*}
\end{definition}
They also gave a pointwise definition of the hyperbolic $p$-sum (see Definition 1.3 in \cite{Li-Xu}).
Then, they calculated the variation of the volume along the hyperbolic $p$-sum
(see Lemma 5.2 in \cite{Li-Xu})
\begin{eqnarray*}
\lim_{t\rightarrow 0+}\frac{\mathrm{Vol}(K+_{p} t \cdot L)-\mathrm{Vol}(K)}{t}=\frac{1}{p}\int_{\mathbb{S}^n}\varphi_{L}^{p}dS_p(K, x),
\end{eqnarray*}
where the horospherical $p$-surface area measure of a smooth uniformly $h$-convex and
compact domain $K\subset \mathbb{H}^{n+1}$ is defined by (see Definition 5.2 \cite{Li-Xu})
\begin{eqnarray*}
dS_p(K, \cdot)=\varphi_{K}^{-p}\det (A[\varphi_K])d\sigma,
\end{eqnarray*}
$u_K=\log \varphi_K$ and $u_L=\log \varphi_L$ are the horospherical support functions of smooth uniformly $h$-convex and
compact domains $K$ and $L$ respectively.
In particular, $p=0$, $dS_p(K, \cdot)$ is just the surface area measure \eqref{Area} of $K$.

Then, Li-Xu proposed the associated horospherical $p$-Minkowski problem (see Problem 5.1 \cite{Li-Xu}):
\begin{problem}\label{LX}
For a given smooth positive function $f(x)$ defined on $\mathbb{S}^n$, we ask the sufficient and necessary conditions for $f(x)$, such that
there exists a smooth, compact and uniformly $h$-convex domain $K\subset \mathbb{H}^{n+1}$ satisfying
\begin{eqnarray*}
dS_p(K, x)=f(x)d\sigma,
\end{eqnarray*}
which is equivalent to find a smooth positive solution $\varphi(x)$ with $A[\varphi(x)]>0$ for all $x \in \mathbb{S}^n$
(or a smooth uniformly $h$-convex solution for short) to the equation
\begin{eqnarray}\label{MA}
\varphi^{-p}(x)\mathrm{det}(A[\varphi(x)])=f(x),
\end{eqnarray}
where $u=\ log \varphi$ is the support function of some horo-convex domain in $\mathbb{H}^{n+1}$.
\end{problem}
In particular, for $p=0$, Problem \ref{LX} is just the prescribed surface area measure problem for
$h$-convex domains in $\mathbb{H}^{n+1}$. For $p=-n$, Problem \ref{LX} is just the prescribed
shifted Gauss curvature problem (see Section 7 in \cite{Li-Xu}) or the $n$-th symmetric
Christoffel problem in $\mathbb{H}^{n+1}$ (see Page 26 in \cite{Esp09}).

Through designing a new volume preserving curvature flow,
Li-Xu solved the existence of normalized even solution to the horospherical $p$-Minkowski problem
for all $p \in \mathbb{R}$. A function $g: \mathbb{S}^n\rightarrow \mathbb{R}$ is called even if
$g(x)=g(-x)$ for all $x \in \mathbb{S}^n$.

\begin{theorem}[Li-Xu \cite{Li-Xu}]\label{Li-X}
Given a smooth, positive and even function $f$ defined on $\mathbb{S}^n$, for any $p \in \mathbb{R}$, there exists a smooth, even
and uniformly $h$-convex solution to the equation
\begin{eqnarray}\label{MA-}
\varphi^{-p}\mathrm{det}(A[\varphi])=\gamma f
\end{eqnarray}
for some $\gamma>0$.
\end{theorem}

Due to the lack of homogeneity of the horospherical $p$-surface area measure, it is
very difficult to remove the normalizing factor $\gamma$ in Theorem \ref{Li-X}. This difficulty also appears
in Orlicz-Minkowski-type
problem \cite{Liu-Lu20} and the Gaussian Minkowski type problem \cite{Huang-Xi21, Liu22, Feng1, Feng2}. In the latter problem,
the degree theory was
used to overcome this difficulty in the even case.

In this paper, we can obtain a non-normalized solution (without the normalizing factor $\gamma$)
to Problem \ref{LX} for $-n\leq p<n$ by the degree theory.

\begin{theorem}\label{Main}
Assume $-n\leq p<n$, there exists a smooth, even and uniformly $h$-convex solution
to the equation \eqref{MA} for any smooth, positive and even function $f$ defined on $\mathbb{S}^n$.
\end{theorem}

It should be noted that the continuity method was not applied to prove existence.
The reason is that the continuity method requires that the corresponding
linearized equation is invertible on any admissible solution and this result is not available for the equation \eqref{MA} studied
in this paper. However, within the framework of the degree theory developed in \cite{Li89} it suffices to know invertibility
on constant solutions which is guaranteed by the uniqueness of even constant solutions to the equation \eqref{MA} for $-n\leq p<n$
(see Theorem 8.1 (6)(7) in \cite{Li-Xu}).
For other ranges $p\geq n$ or $p<-n$, either the constant solution may not be unique or its uniqueness is unknown.
(see Theorem 8.1 in \cite{Li-Xu}). So, it is difficult to
prove existence by the degree theory,
although the a prior estimates for solutions to the equation \eqref{MA} can be established for all $p<n$.

\begin{remark}
It is interesting to consider the uniqueness of solutions to the equation \eqref{MA} when $f$ is not constant. In contrast with
$L_p$ Minkowski problem in Euclidean space for which the uniqueness of solutions is
guaranteed by $L_p$ Brunn-Minkowski inequality for $p\geq 1$ \cite{Lu93}, the uniqueness of solutions to the horospherical $p$-Minkowski problem
\ref{LX} is unknown due to the lack of analogous inequality in hyperbolic space (see related discussions in  Section 9 \cite{Li-Xu}).
\end{remark}

The present paper is built up as follows. In Sect. 2, we obtain the a priori estimates.  We
will prove Theorem \ref{Main} in Sect. 3.

\section{The a priori estimates}

For convenience, in the following of this paper, we always assume that $f$ is
a smooth, positive and even function on $\mathbb{S}^n$, and $\varphi$ is a smooth, even and uniformly $h$-convex solution to the equation \eqref{MA}.
Moreover, let $\Omega$ be the smooth, compact and uniformly $h$-convex domain in $\mathbb{H}^{n+1}$ with the horospherical support function $u=\log \varphi$.
Clearly, $\Omega$ is symmetric about the origin and $\varphi(x)> 1$ for $x \in \mathbb{S}^n$.

The following easy and important equality is key for
the $C^0$ estimate.
\begin{lemma}
We have
\begin{eqnarray}\label{c0-12}
\frac{1}{2}\Big(\max_{\mathbb{S}^n}\varphi+\frac{1}{\max_{\mathbb{S}^n}\varphi}\Big)\leq \min_{\mathbb{S}^n}\varphi.
\end{eqnarray}
\end{lemma}

\begin{proof}
The inequality can be found in the proof of Lemma 7.2 in \cite{Li-Xu}. For completeness, we give a proof here.
Assume that $\varphi(x_1)=\max_{\mathbb{S}^n}\varphi$ and denote $\overline{X}(x_1)=G^{-1}(x_1)$ as before. Then, we have for any $x \in \mathbb{S}^n$ by
the definition of the horospherical support function \eqref{SD}
\begin{eqnarray*}
-\langle \overline{X}(x_1), (x, 1)\rangle\leq \varphi(x), \quad \forall x \in \mathbb{S}^n.
\end{eqnarray*}
Substituting the expression \eqref{X} for $\overline{X}$ into the above equality yields
\begin{eqnarray}\label{D0-3}
\frac{1}{2}\varphi(x_1)(1+\langle x_1, x\rangle)+\frac{1}{2}\frac{1}{\varphi(x_1)}\Big(1-\langle x_1, x\rangle\Big)\leq \varphi(x),
\end{eqnarray}
where we used the fact $D\varphi(x_1)=0$. Note that $\varphi(x_1)\geq 1$, we find from \eqref{D0-3}
\begin{eqnarray}\label{D0-4}
\frac{1}{2}\Big(\varphi(x_1)+\frac{1}{\varphi(x_1)}\Big)\leq \varphi(x) \quad \mbox{for} \quad \langle x, x_1\rangle\geq 0.
\end{eqnarray}
Since $\varphi$ is even, we can assume
that the minimum point $x_0$ of $\varphi(x)$ satisfies $\langle x_0, x_1\rangle\geq 0$. Thus,
the equality \eqref{c0-12} follows that from \eqref{D0-4}.
\end{proof}

\begin{remark}
It is very interesting to compare the inequality \eqref{c0-12} with the the analogous inequality for convex bodies in
Euclidean space. In fact,
for an origin-symmetric convex body in Euclidean space, the definition of its support function gives
\begin{eqnarray}\label{E}
|\langle x_{1}, x_{0}\rangle|\max_{\mathbb{S}^n}h \leq \min_{\mathbb{S}^n}h,
\end{eqnarray}
where $h$ is the support function of the convex body, $x_1$ and $x_0$ are the maximum and minimum points of
$h$ respectively. Clearly, the inequality \eqref{c0-12} in hyperbolic space is better than \eqref{E}, since
 the term $|\langle x_{1}, x_{0}\rangle|$ in \eqref{E} may not be bounded from below.
\end{remark}

Now we begin to consider the $C^0$-estimate. A similar estimate is obtained in Lemma 7.2 \cite{Li-Xu} when
the volume of $h$-convex domain is fixed. But for our case, the volume is not fixed,
we use the maximum principle to get the $C^0$-estimate.
\begin{lemma}\label{C-C0}
If $p<n$, we have
\begin{eqnarray}\label{c0}
0<\frac{1}{C}\leq u(x)\leq C, \quad \forall \ x \in \mathbb{S}^n,
\end{eqnarray}
where $C$ is a positive constant depending on $p$, $n$ and $f$.
\end{lemma}

\begin{proof}
Using the maximum principle, we have from the equation \eqref{MA}
\begin{eqnarray*}
(\max_{\mathbb{S}^n}\varphi)^{n-p}\frac{1}{2^n}\bigg[1-\frac{1}{(\max_{\mathbb{S}^n}\varphi)^2}\bigg]^n\geq C>0
\end{eqnarray*}
and
\begin{eqnarray*}
(\min_{\mathbb{S}^n}\varphi)^{n-p}\frac{1}{2^n}\bigg[1-\frac{1}{(\min_{\mathbb{S}^n}\varphi)^2}\bigg]^n\leq C.
\end{eqnarray*}
Since the function
\begin{eqnarray*}
g(x)=x^{n-p}\frac{1}{2^n}\bigg(1-\frac{1}{x^2}\bigg)^n
\end{eqnarray*}
is increasing in $[1, +\infty)$ for $p<n$, $g(1)=0$ and $g(+\infty)=+\infty$, we obtain
\begin{eqnarray}\label{c0-11}
\min_{\mathbb{S}^n}\varphi\leq C, \quad \mbox{and} \quad \max_{\mathbb{S}^n}\varphi\geq C>1.
\end{eqnarray}
Combining \eqref{c0-11} and \eqref{c0-12}, we find
\begin{eqnarray*}
1<C\leq\min_{\mathbb{S}^n}\varphi\leq \max_{\mathbb{S}^n}\varphi\leq C^{\prime},
\end{eqnarray*}
which implies that
\begin{eqnarray*}
0<\frac{1}{C}\leq\min_{\mathbb{S}^n}u\leq \max_{\mathbb{S}^n}u\leq C.
\end{eqnarray*}
So, we complete the proof.
\end{proof}

As a corollary, we have the gradient estimate from Lemma 7.3 in \cite{Li-Xu}.

\begin{corollary}\label{C-C1}
We have
\begin{eqnarray}\label{c1}
|D\varphi(x)|\leq C, \quad \forall \ x \in \mathbb{S}^n,
\end{eqnarray}
where $C$ is a positive constant depending only on the constant in Lemma \ref{C-C0}.
\end{corollary}

We give some notations before considering the $C^2$ estimate.
Denote by
\begin{eqnarray*}
U_{ij}=\varphi_{ij}-\frac{1}{2}\frac{|D\varphi|^2}{\varphi}\delta_{ij}+\frac{1}{2}(\varphi-\frac{1}{\varphi})\delta_{ij}
\end{eqnarray*}
and
\begin{eqnarray*}
F(U)=(\det U)^{\frac{1}{n}}, \quad F^{ij}=\frac{\partial F}{\partial U_{ij}}, \quad F^{ij, kl}=\frac{\partial^2 F}{\partial U_{ij}\partial U_{kl}}.
\end{eqnarray*}

\begin{lemma}\label{C2}
We have for $1\leq i\leq n$
\begin{eqnarray}\label{c2}
\lambda_i(U(x))\leq C, \quad \forall \ x \in \mathbb{S}^n,
\end{eqnarray}
where $\lambda_1(U), ..., \lambda_n(U)$ are eigenvalues of the matrix $U$ and $C$ is a positive constant depending
only on the constant in Lemma \ref{C-C0} and Corollary \ref{C-C1}.
\end{lemma}

\begin{proof}
Since
\begin{eqnarray*}
\lambda_i(U)\leq \mathrm{tr} U=\Delta \varphi-\frac{n}{2\varphi}|D\varphi|^2+\frac{n}{2}(\varphi-\frac{1}{\varphi}), \quad \forall 1\leq i\leq n,
\end{eqnarray*}
it is sufficient to prove $\Delta\varphi\leq C$ in view of $C^0$ estimate \eqref{c0} and $C^1$ estimate \eqref{c1}.
Moreover, these two estimates together with the positivity of the matrix $U$ imply $\lambda_{i}(D^2\varphi)\geq -C$ for $1\leq i\leq n$.
Thus,
\begin{eqnarray}\label{C2-3}
|\lambda_i(D^2\varphi)|\leq C\max \{ \max_{\mathbb{S}^n}\Delta \varphi, 1\} \quad \forall 1\leq i\leq n.
\end{eqnarray}
We take the auxiliary function
$$W(x)=\Delta \varphi.$$
Assume $x_0$ is the maximum point of $W$. After an appropriate
choice of the normal frame at $x_0$, we further assume $U_{ij}$, hence $\varphi_{ij}$ and $F^{ij}$ is diagonal at the point $x_0$. Then,
\begin{equation}\label{W1}
W_i(x_0)=\sum_{k}\varphi_{kki}=0,
\end{equation}
and
\begin{equation}\label{W2}
W_{ii}(x_0)=\sum_{k}\varphi_{kkii}\leq0.
\end{equation}
From the positivity of $F^{ij}$ and \eqref{W2}, we arrive at $x_0$ if $\Delta \varphi$ is large enough
\begin{eqnarray*}
0&\ge&\sum_iF^{ii}W_{ii}=\sum_iF^{ii}\sum_{k}\varphi_{kkii}\geq\sum_iF^{ii}\sum_{k}\big(\varphi_{iikk}-C\Delta \varphi\big),
\end{eqnarray*}
where we use Ricci identity and the equality \eqref{C2-3} to get the last inequality. Thus it
follows from the definition of $U$, \eqref{c0}, \eqref{c1} and \eqref{W1},
\begin{eqnarray}\label{C2-1}
0&\ge&\sum_iF^{ii}\sum_{q}\bigg[U_{iiqq}+\Big(\frac{1}{2\varphi}|D\varphi|^2\Big)_{qq}-
\frac{1}{2}\Big(\varphi-\frac{1}{\varphi}\Big)_{qq}\bigg]
-C\sum_{i}F^{ii}\Delta \varphi\nonumber\\
&\ge&\sum_iF^{ii}\sum_{q}\bigg[U_{iiqq}+\frac{1}{\varphi}(\varphi_{qq})^2-C\Delta\varphi-C\bigg]
-C\sum_iF^{ii}\Delta \varphi.
\end{eqnarray}
Differentiating the equation \eqref{MA} twice gives
\begin{eqnarray*}
F^{ii}U_{iiqq}+F^{ij, kl}U_{ijq}U_{klq}=\Big[(\varphi^pf)^{\frac{1}{n}}\Big]_{qq},
\end{eqnarray*}
which yields
\begin{eqnarray}\label{C2-2}
F^{ii}U_{iiqq}\geq -C\Delta \varphi-C
\end{eqnarray}
in view of \eqref{c0} and \eqref{c1}.
Substituting \eqref{C2-2} into \eqref{C2-1} and using $(\Delta \varphi)^2 \leq n \sum_{q}(\varphi_{qq})^2$, we have
\begin{eqnarray}\label{QQ}
0\ge\bigg(C(\Delta \varphi)^2-C\Delta \varphi-C\bigg)\sum_{i}F^{ii}-C\Delta \varphi-C.
\end{eqnarray}
Note that
\begin{eqnarray}\label{C2-8}
\sum_{i}F^{ii}\geq 1.
\end{eqnarray}
Then we conclude at $x_0$ by combining the inequalities \eqref{QQ} and \eqref{C2-8}
\begin{eqnarray*}
C\geq |\Delta \varphi|^2
\end{eqnarray*}
if $\Delta \varphi$ is chosen large enough.
So, we complete the proof.
\end{proof}

\section{The proof of the main theorem}

In this section, we use the degree theory for nonlinear elliptic
equations developed in \cite{Li89} to prove Theorem \ref{Main}. Such approach was also used
in prescribed curvature problems for star-shaped hypersuraces \cite{An, Li02, Jin, Li-Sh}, prescribed curvature problems for
convex hypersurfaces
\cite{Guan02, Guan21} and the Gaussian
Minkowski type problem \cite{Huang-Xi21, Liu22, Feng1, Feng2}.

For the use of the degree theory, the uniqueness of constant solutions to the equation
\eqref{MA} is important for us. Fortunately, this fact can be guaranteed by Theorem 8.1 (6)(7) in \cite{Li-Xu}
for $-n\leq p<n$. We summarize it as follows.

\begin{lemma}\label{U-C}
For $-n\leq p<n$, there exists a unique even solution $\varphi=c$
to the equation
\begin{eqnarray}\label{MA-c}
\varphi^{-p}(x)\mathrm{det}(A[\varphi(x)])=\gamma
\end{eqnarray}
for any $\gamma>0$, where $c$ is the unique positive solution to
the equation
\begin{eqnarray}\label{con}
c^{-p}\Big(\frac{1}{2}(c-c^{-1})\Big)^n=\gamma.
\end{eqnarray}
\end{lemma}

\begin{proof}
For $-n<p<n$,
Theorem 8.1 (6) in \cite{Li-Xu} tells us that there exist a unique solution $\varphi=c$
to the equation \eqref{MA-c} for any $\gamma>0$, where $c$ is the unique positive solution to the equation \eqref{con}.

For $p=-n$, Theorem 8.1 (7) in \cite{Li-Xu} tells us that the solutions to
the equation \eqref{MA-c} for any $\gamma>0$ are given by
\begin{eqnarray*}
\varphi(x)=\Big(1+2\gamma^{\frac{1}{n}}\Big)^{\frac{1}{2}}\Big(\sqrt{|x_0|^2+1}-\langle x_0, x\rangle\Big),
\end{eqnarray*}
where $x_0 \in \mathbb{R}^{n+1}$. Clearly, $\varphi(x)=\Big(1+2\gamma^{\frac{1}{n}}\Big)^{\frac{1}{2}}$ is the unique even solution.
Thus, the conclusion holds true.
\end{proof}

Now, we begin to prove Theorem \ref{Main}.
After establishing the  a priori estimates \eqref{c0}, \eqref{c1} and \eqref{c2} for $p< n$, we know from the equation \eqref{MA}
\begin{eqnarray}\label{C2+++}
\lambda_i(U(x))\geq C>0, \quad \forall \ x \in \mathbb{S}^n,
\end{eqnarray}
where $\lambda_1(U), ..., \lambda_n(U)$ are eigenvalues of the matrix $U$. Thus,
the equation \eqref{MA} is uniformly elliptic. From Evans-Krylov estimates \cite{Eva82, Kry83}, and Schauder estimates \cite{GT}, we have
\begin{eqnarray}\label{C2+}
|\varphi|_{C^{4,\alpha}(\mathbb{S}^n)}\leq C
\end{eqnarray}
for any smooth, even and uniformly $h$-convex solution $\varphi$ to the equation \eqref{MA}.
We define
\begin{eqnarray*}
\mathcal{B}^{2,\alpha}(\mathbb{S}^n)=\{\varphi \in
C^{2,\alpha}(\mathbb{S}^n): \varphi \ \mbox{is even}\}
\end{eqnarray*}
and
\begin{eqnarray*}
\mathcal{B}_{0}^{4,\alpha}(\mathbb{S}^n)=\{\varphi \in
C^{4,\alpha}(\mathbb{S}^n): A[\varphi]>0 \ \mbox{and} \ \varphi \ \mbox{is even}\}.
\end{eqnarray*}
Let us consider a family of the mappings for $0\leq t\leq 1$
\begin{eqnarray*}
\mathcal{L}(\cdot, t): \mathcal{B}_{0}^{4,\alpha}(\mathbb{S}^n)\rightarrow
\mathcal{B}^{2,\alpha}(\mathbb{S}^n),
\end{eqnarray*}
which is defined by
\begin{eqnarray*}
\mathcal{L}(\varphi, t)=\det U-\varphi^p [(1-t)\gamma+t f],
\end{eqnarray*}
where and the constant $\gamma$ will be chosen later and $U$ is denoted as before
\begin{eqnarray*}
U=D^2\varphi-\frac{1}{2}\frac{|D\varphi|^2}{\varphi}I+\frac{1}{2}\Big(\varphi-\frac{1}{\varphi}\Big)I.
\end{eqnarray*}

Let $$\mathcal{O}_R=\{\varphi \in \mathcal{B}_{0}^{4,\alpha}(\mathbb{S}^n):
1+\frac{1}{R}< \varphi, \ \frac{1}{R} I< U, \ |\varphi|_{C^{4,\alpha}(\mathbb{S}^n)}<R\},$$ which clearly is an open
set of $\mathcal{B}_{0}^{4,\alpha}(\mathbb{S}^n)$. Moreover, if $R$ is
sufficiently large, $\mathcal{L}(\varphi, t)=0$ has no solution on $\partial
\mathcal{O}_R$ by the a priori estimates established in \eqref{c0}, \eqref{C2+++} and \eqref{C2+}.
Therefore the degree $\deg(\mathcal{L}(\cdot, t), \mathcal{O}_R, 0)$ is
well-defined for $0\leq t\leq 1$. Using the homotopic invariance of
the degree (Proposition 2.2 in \cite{Li89}), we have
\begin{eqnarray}\label{hot}
\deg(\mathcal{L}(\cdot, 1), \mathcal{O}_R, 0)=\deg(\mathcal{L}(\cdot, 0), \mathcal{O}_R, 0).
\end{eqnarray}

For $-n\leq p<n$, Lemma \ref{U-C} tells us that
$\varphi=c$ is the unique even solution for $\mathcal{L}(\varphi, 0)=0$ in $\mathcal{O}_R$.
Direct calculation show that the linearized operator $L_c$ of $\mathcal{L}$ at
$\varphi=c$ is
\begin{eqnarray*}
\Big[\frac{1}{2}\Big(c-\frac{1}{c}\Big)\Big]^{1-n}L_{c}(\psi)=
\Big(\Delta_{\mathbb{S}^n}+\frac{1}{2}\Big(n-p+\frac{n+p}{c^2}\Big)\Big)\psi,
\end{eqnarray*}
where $\gamma$ is given by the equation \eqref{con}.
Since spherical Laplacian has a discrete spectrum, we choose $c=c_0$
such that $L_{c_0}$ is an invertible operator for $p\neq -n$. For $p=-n$,
$\Big[\frac{1}{2}\Big(c-\frac{1}{c}\Big)\Big]^{1-n}L_{-n}(\psi)=\Delta_{\mathbb{S}^n}\psi+n\psi=0$
has the unique even solution $\psi=0$.
Thus, $L_{c_0}$ is an invertible operator for $p$. So, we have by Proposition
2.3 in \cite{Li89}
\begin{eqnarray*}
\deg(\mathcal{L}(\cdot, 0), \mathcal{O}_R, 0)=\deg(L_{c_0}, \mathcal{O}_R, 0).
\end{eqnarray*}

Because
the eigenvalues of the Beltrami-Laplace operator $\Delta$ on $\mathbb{S}^n$ are strictly less than
$-n$ except for the first two eigenvalues $0$ and $-n$,
there is only one positive eigenvalue $\frac{1}{2}\Big(n-p+\frac{n+p}{c^2}\Big)$ of $\Big[\frac{1}{2}\Big(c-\frac{1}{c}\Big)\Big]^{1-n}L_{c}$
with multiplicity $1$ if we note that $\frac{1}{2}\Big(n-p+\frac{n+p}{c^2}\Big)-n\leq0$ for $p\geq-n$.
Then we have by Proposition
2.4 in \cite{Li89}
\begin{eqnarray*}
\deg(\mathcal{L}(\cdot, 0), \mathcal{O}_R, 0)=\deg(L_{c_0}, \mathcal{O}_R, 0)=-
1.
\end{eqnarray*}
Therefore, it follows from \eqref{hot}
\begin{eqnarray*}
\deg(\mathcal{L}(\cdot, 1), \mathcal{O}_R; 0)=\deg(\mathcal{L}(\cdot, 0), \mathcal{O}_R, 0)=-
1.
\end{eqnarray*}
So, we obtain a solution at $t=1$. This completes the proof of
Theorem \ref{Main}.

\textbf{Acknowledgement:} The author would
like to express his gratitude to Dr. Botong Xu for
pointing some mistakes on this paper.

\end{document}